\documentclass[12pt]{article}

\usepackage{mathrsfs}
\usepackage{indentfirst}
\usepackage{amsmath,amssymb,amsfonts,amsthm,graphics,graphicx,pifont}
\usepackage{makeidx}
\usepackage{float}
\usepackage{color}
\usepackage[font=small,labelfont=bf,labelsep=none]{caption}
\usepackage{epstopdf}

\usepackage{color}
\usepackage{subfigure}
\usepackage[noend]{algpseudocode}
\usepackage{algorithmicx,algorithm}
\graphicspath{{\figs}}

\graphicspath{{\figs}}

\begin{document}

\newtheorem{theorem}{Theorem}[section]
\newtheorem{observation}[theorem]{Observation}
\newtheorem{corollary}[theorem]{Corollary}
\newtheorem{problem}[theorem]{Problem}
\newtheorem{question}[theorem]{Question}
\newtheorem{lemma}[theorem]{Lemma}
\newtheorem{proposition}[theorem]{Proposition}

\theoremstyle{definition}
\newtheorem{definition}[theorem]{Definition}
\newtheorem{guess}[theorem]{Conjecture}
\newtheorem{claim}[theorem]{Claim}
\newtheorem{example}[theorem]{Example}
\newtheorem{remark}[theorem]{Remark}

\makeatletter
  \newcommand\figcaption{\def\@captype{figure}\caption}
  \newcommand\tabcaption{\def\@captype{table}\caption}
\makeatother

\newtheorem{acknowledgement}[theorem]{Acknowledgement}

\newtheorem{axiom}[theorem]{Axiom}
\newtheorem{case}[theorem]{Case}
\newtheorem{conclusion}[theorem]{Conclusion}
\newtheorem{condition}[theorem]{Condition}
\newtheorem{conjecture}[theorem]{Conjecture}
\newtheorem{criterion}[theorem]{Criterion}
\newtheorem{exercise}[theorem]{Exercise}
\newtheorem{notation}[theorem]{Notation}
\newtheorem{solution}[theorem]{Solution}
\newtheorem{summary}[theorem]{Summary}
\newtheorem{fact}[theorem]{Fact}

\newcommand{\pp}{{\it p.}}
\newcommand{\de}{\em}
\newcommand{\mad}{\rm mad}

\newcommand*{\QEDA}{\hfill\ensuremath{\blacksquare}}  
\newcommand*{\QEDB}{\hfill\ensuremath{\square}}  

\newcommand{\qf}{Q({\cal F},s)}
\newcommand{\qff}{Q({\cal F}',s)}
\newcommand{\qfff}{Q({\cal F}'',s)}
\newcommand{\f}{{\cal F}}
\newcommand{\ff}{{\cal F}'}
\newcommand{\fff}{{\cal F}''}
\newcommand{\fs}{{\cal F},s}

\newcommand{\g}{\gamma}
\newcommand{\wrt}{with respect to }

\def\C#1{|#1|}
\def\E#1{|E(#1)|}
\def\V#1{|V(#1)|}
\def\cB{{\mathcal B}}
\def\cD{{\mathcal D}}
\def\cF{{\mathcal F}}
\def\cI{{\mathcal I}}
\def\cP{{\mathcal P}}
\def\cT{{\mathcal T}}
\def\NN{{\mathbb N}}
\def\VEC#1#2#3{#1_{#2},\dots,#1_{#3}}
\def\VECOP#1#2#3#4{#1_{#2}#4\dots#4#1_{#3}}
\def\vp{\varphi}
\def\ve{\varepsilon}
\def\FR{\frac}
\def\st{\colon\,}
\def\esub{\subseteq}
\def\nul{\varnothing}
\def\SE#1#2#3{\sum_{#1=#2}^{#3}}
\def\Phi{\vp}
\def\iarb{\Upsilon}
\def\ipac{\nu}

\title{Rainbow monochromatic $k$-edge-connection colorings of graphs\footnote{Supported by NSFC No.11871034 and 11531011.}}

\author{\small Ping Li, ~ Xueliang Li\\
\small Center for Combinatorics and LPMC\\
\small Nankai University\\
\small Tianjin 300071, China\\
\small wjlpqdxs@163.com, ~ lxl@nankai.edu.cn\\
}
\date{}
\maketitle

\vspace{-2pc}

\begin{abstract}
A path in an edge-colored graph is called a monochromatic path if all edges of the path have a same color.
We call $k$ paths $P_1,\cdots,P_k$ rainbow monochromatic paths if every $P_i$ is monochromatic and for any
two $i\neq j$, $P_i$ and $P_j$ have different colors. An edge-coloring of a graph $G$ is said to be a rainbow
monochromatic $k$-edge-connection coloring (or $RMC_k$-coloring for short) if every two distinct vertices
of $G$ are connected by at least $k$ rainbow monochromatic paths. We use $rmc_k(G)$ to denote the maximum number
of colors that ensures $G$ has an $RMC_k$-coloring, and this number is called the rainbow monochromatic
$k$-edge-connection number. We prove the existence of $RMC_k$-colorings of graphs,
and then give some bounds of $rmc_k(G)$ and present some graphs whose $rmc_k(G)$ reaches the lower bound.
We also obtain the threshold function for $rmc_k(G(n,p))\geq f(n)$, where $\lfloor\frac{n}{2}\rfloor> k\geq 1$.\\
[2mm] {\bf Keywords:} monochromatic path, rainbow monochromatic paths, rainbow monochromatic
$k$-edge-connection coloring (number), threshold function.\\
[2mm] {\bf AMS subject classification (2010)}: 05C15, 05C40.
\end{abstract}

\baselineskip16pt
\section{Introduction}

All graphs considered in this paper are simple, except for some graphs in Section $3$.
Let $G$ be a graph and let $V(G)$, $E(G)$ denote the vertex set and the edge set of $G$, respectively.
Let $|G|$ (also $v(G)$) denote the number of vertices of $G$ (the {\em order} of $G$),
and let $e(G)$ (also $||G||$) denote the number of edges of $G$ (the {\em size} of $G$).
If there is no confusion, we use $n$ and $m$ to denote, respectively,
the number of vertices and edges of a graph, throughout this paper.
For $v\in V(G)$, let $d_G(v)$ denote the degree of $v$, $N(v)$ denote the set of neighbors of $v$, and $N[v]=N(v)\cup\{v\}$.
Let $\delta(G)$ and $\Delta(G)$ denote the minimum and maximum degree of $G$, respectively.
Let $U$ and $S$ be a vertex set and an edge set of $G$, respectively. then,
$G-U$ is a graph obtained from $G$ by deleting the vertices of $U$ together with the edges incident with the vertices of $U$, and
$G-S$ is a graph obtained from $G$ by deleting the edges of $S$, and then deleting the isolate vertices.
Let $G[U]$ and $G[S]$ be the vertex-induced and edge-induced subgraph of $G$, respectively, by $U$ and $S$.
The distance of $u, v$ in $G$ is denoted by $d_G(u,v)$.
For all other terminology and notation not defined here we follow Bondy and Murty \cite{BM}.

For a graph $G$, let $\Gamma:~E(G)\rightarrow[k]$ be an edge-coloring of $G$ that allows a same color to be assigned to adjacent edges,
here and in what follows $[k]$ denotes the set $\{1,2,\cdots, k\}$ of integers for a positive integer $k$.
For an edge $e$ of $G$, we use $\Gamma(e)$ to denote the color of $e$.
If $H$ is a subgraph of $G$, we also use $\Gamma(H)$ to denote the set of colors on the edges of $H$ and use $|\Gamma(H)|$
to denote the number of colors in $\Gamma(H)$.

A {\em monochromatic uv-path} is a $uv$-path of $G$ whose edges colored with a same color, and $G$ is {\em monochromatically connected }
if for any two vertices of $G$, $G$ has a monochromatic path connecting them.
An edge-coloring $\Gamma$ of $G$ is a {\em monochromatic connection coloring (or MC-coloring for short) }if it makes $G$ monochromatically connected.
The {\em monochromatic connection number} of a connected graph $G$, denoted by $mc(G)$, is the maximum number of colors that are allowed in order to
make $G$ monochromatically connected. An {\em extremal MC-coloring} of $G$ is an $MC$-coloring that uses $mc(G)$ colors.

The notion monochromatic connection coloring was introduced by Caro and Yuster \cite{CY}. Some results were obtained in \cite{CLW,GLQZ,JLW,MWYY,LW}.
Later, Gonzlez-Moreno et al. in \cite{MGM} generalized the above concept to digraphs.

We list the main results in \cite{CY} below.
\begin{theorem} [\cite{CY}] \label{RMD-CY}
Let $G$ be a connected graph with $n\geq 3$. If $G$ satisfies any of the following properties, then $mc(G)=m-n+2$.
\begin{enumerate}
 \item $\overline{G}$ (the complement of $G$) is a $4$-connected graph;
 \item $G$ is triangle-free;
 \item $\Delta(G)<n-\frac{2m-3(n-1)}{n-3}$;
 \item $diam(G)\geq3$;
 \item $G$ has a cut vertex.
 \end{enumerate}
\end{theorem}

The Erd$\ddot{o}$s-R$\acute{e}$nyi random graph model $G(n,p)$ will be studied in this paper. The graph $G(n,p)$ is defined on $n$ labeled vertices
(informally, we use $[n]$ to denote the $n$ labeled vertices) in which each edge is chosen independently and randomly with probability $p$.
A {\em property} of graphs is a subset of the set of all graphs on $[n]$ (such as connectivity, minimum degree, et al).
If a property $Q$ has $Pr[Q]\rightarrow 1$ when $n\rightarrow +\infty$, then we call the property $Q$ {\em almost surely}.
A property $Q$ is {\em monotone increasing} if whenever $H$ is a graph obtained from $H'$ by adding some addition edges and $H'$ has property $Q$,
then $H$ also has the property $Q$.

Given two functions $a(n)$ and $b(n)$ with $a(n)\geq 0$ and $b(n)>0$, we write $a(n)=o(b(n))$ if $a(n)/b(n)\rightarrow 0$ when $n\rightarrow \infty$;
$a(n)=O(b(n))$ if there is a constant $C$ such that $a(n)\leq Cb(n)$ for all $n$;
and finally $a(n)=\omega(b(n))$ if $b(n)=o(a(n))$.

A function $h(n)$ is a {\em threshold function} for an increasing property $Q$, if for any two functions $h_1(n)=o(h(n))$ and $h(n)=o(h_2(n))$,
$G(n,h_1(n))$ does not have property $Q$ almost surely and $G(n,h_2(n))$ has property $Q$ almost surely.
Moreover, $h(n)$ is called a {\em sharp threshold function} of $Q$ if there exist two positive constants $c_1$ and $c_2$ such that
$G(n,p(n))$ does not have property $Q$ almost surely when $p(n)\leq c_1h(n)$ and $G(n,p(n))$ has property $Q$ almost surely when $p(n)\geq c_2h(n)$.
It was proved in \cite{EG} that every monotone increasing graph property has a sharp threshold function.
The property monochromatic connection coloring of a graph (and also the properties monochromatic $k$-edge-connection coloring,
uniformly monochromatic $k$-edge-connection coloring and rainbow monochromatic $k$-edge-connection coloring of graphs which are defined later)
is monotone increasing, and therefore it has a sharp threshold function.

\begin{theorem}[\cite{GLQZ}]
Let $f(n)$ be a function satisfying $1\leq f(n)<{n\choose2}$. Then
$$p=
\begin{cases}
\frac{f(n)+n\log \log n}{n^2}, & \mbox{ if } ln\log n\leq f(n)<{n\choose2},\mbox{ where }l\in \mathbb{R}^+;\\
\frac{\log n}{n}, &\mbox{ if }f(n)=o(n\log n).
\end{cases}$$
is a sharp threshold function for the property $mc(G(n,p))\geq f(n)$.
\end{theorem}

Now we generalize the concept monochromatic connection coloring of graphs. There are three ways to generalize this concept.

The first generalized concept is called the {\em monochromatic $k$-edge-connection coloring} (or {\em $MC_k$-coloring} for short) of $G$,
which requires that every two distinct vertices of $G$ are connected by at least $k$ edge-disjoint monochromatic paths (allow some of the paths to have different colors).
The {\em monochromatically $k$-edge-connection number} of a connected $G$, denoted by $mc_k(G)$, is the maximum number of colors that are allowed in order to make
$G$ monochromatically $k$-edge-connected.

The second generalized concept is called the {\em uniformly monochromatic $k$-edge-connection coloring} (or {\em $UMC_k$-coloring} for short) of $G$,
which requires that every two distinct vertices of $G$ are connected by at least $k$ edge-disjoint monochromatic paths such that all
these $k$ paths have the same color (note that for different pairs of vertices the paths may have different colors).
The {\em uniformly monochromatically $k$-edge-connection number} of  a connected $G$, denoted by $umc_k(G)$, is the maximum number of colors that are allowed in order to
make $G$ uniformly monochromatically $k$-edge-connected. These two concepts were studied in \cite{LL}.

It is obvious that a graph has an $MC_k$-coloring (or $UMC_k$-coloring) if and only if $G$ is $k$-edge-connected.
We mainly study the third generalized concept in this paper, which is called the {\em rainbow monochromatic $k$-edge-connection coloring}
(or {\em $RMC_k$-coloring} for short) of a connected graph. One can see later, compare the results for $MC$-colorings, $MC_k$-colorings, $UMC_k$-colorings and $RMC_k$-colorings of graphs,
the concept $RMC_k$-coloring has the best form among all the generalized concepts of the $MC$-coloring.

The definition of the third generalized concept goes as follows. For an edge-colored simple graph $G$
(if $G$ has parallel edges but no loops, the following notions are also reasonable),
if for any two distinct vertices $u$ and $v$ of $G$, $G$ has $k$ edge-disjoint monochromatic paths connecting them,
and the colors of these $k$ paths are pairwise differently, then we call such $k$ monochromatic paths {\em $k$ rainbow monochromatic $uv$-paths}.
An edge-colored graph is {\em rainbow monochromatically $k$-edge-connected} if every two vertices of the graph are connected by at least $k$ rainbow monochromatic paths
in the graph. An edge-coloring $\Gamma$ of a connected graph $G$ is a {\em rainbow monochromatic k-edge-connection coloring} (or {\em $RMC_k$-coloring} for short) if it makes $G$ rainbow
monochromatically $k$-edge-connected. The {\em rainbow monochromatically k-edge-connection number} of a connected graph $G$, denoted by $rmc_k(G)$, is the maximum number of colors
that are allowed in order to make $G$ rainbow monochromatically $k$-edge-connected. An {\em extremal $RMC_k$-coloring} of $G$ is an $RMC_k$-coloring that uses
$rmc_k(G)$ colors.

If $k=1$, then an $RMC_k$-coloring (also $MC_k$-coloring and $UMC_k$-coloring) is reduced to a monochromatic connection coloring for any connected graph.

In an edge-colored graph $G$, if a color $i$ only color one edge of $E(G)$, then we call the color $i$ a {\em trivial color}, and call the edge
(tree) a {\em trivial edge} ({\em trivial tree}). Otherwise we call the edges (colors, trees) {\em nontrivial}.
A subgraph $H$ of $G$ is called an {\em $i$-induced subgraph} if $H$ is induced by all the edges of $G$ with the same color $i$. Sometimes, we also
call $H$ a {\em color-induced subgraph}.

If $\Gamma$ is an extremal $RMC_k$-coloring of $G$, then each color-induced subgraph is connected.
Otherwise we can recolor the edges in one of its components by a fresh color, then the new edge-coloring is also
an $RMC_k$-coloring of $G$, but the number of colors is increased by one, which contradicts that $\Gamma$ is extremal.
Furthermore, each color-induced subgraph does not have cycles; otherwise we can recolor one edge in a cycle by a fresh color.
Then the new edge-coloring is also an $RMC_k$-coloring of $G$,
but the number of colors is increased, a contradiction. Therefore, we have the following result.

\begin{proposition}
If $\Gamma$ is an extremal $RMC_k$-coloring of $G$, then each color-induced subgraph is a tree.
\end{proposition}

If $\Gamma$ is an extremal $RMC_k$-coloring of $G$ for $i\in \Gamma(G)$, we call an $i$-induced subgraph of $G$
an {\em $i$-induced tree} or a {\em color-induced tree}. We also call it a tree sometimes if there is no confusion.

The paper is organized as follows. Section $2$ will give some preliminary results. In Section $3$,
we study the existence of $RMC_k$-colorings of graphs. In Section $4$, we give some bounds of $rmc_k(G)$,
and present some graphs whose $rmc_k(G)$ reaches the lower bound. In Section $5$, we obtain the threshold function for
$rmc_k(G)\geq f(n)$, where $\lfloor\frac{n}{2}\rfloor> k\geq 1$.

\section{Preliminaries}

Suppose that $a=(a_1,\cdots,a_q)$ and $b=(b_1,\cdots,b_p)$ are two positive integer sequences
whose lengths $p$ and $q$ may be different. Let $\prec$ be the {\em lexicographic order} for integer sequences,
i.e., $a\prec b$ if for some $h\geq1$, $a_j=b_j$ for $j< h$ and $a_h<b_h$, or $p>q$ and $a_j=b_j$ for $j\leq q$.

Let $D,n,s$ be integers with $n\geq 5$ and $1\leq s\leq n-4$.
Let $r$ be an integer satisfying $D<r{n-s\choose 2}$.
For an integer $t\geq r$,
suppose $f(\mathbf{x}_t)=f(x_1,\cdots,x_t)=\sum_{i\in[t]}{x_i-1\choose 2}$ and $g(\mathbf{x}_t)=g(x_1,\cdots,x_t)=\sum_{i\in[t]}(x_i-2)$,
where $x_i\in\{3,4,\cdots,n-s\}$. We use $\mathcal{S}_t$ to denote the set of optimum solutions of the following problem:
\begin{align*}
&\min & &g(\mathbf{x}_t)\\
&s.t. & &f(\mathbf{x}_t)\geq D \mbox{ and }x_i\in\{3,\cdots,n-s\}\mbox{ for each }i\in [t].
\end{align*}

\begin{lemma} \label{expression}
There are integers $r,x$ with $r\leq t$ and $3\leq x< n-s$, such that the above problem has a solution $\mathbf{x}_t=(x_1,\cdots,x_t)$
in $\mathcal{S}_t$ satisfying that $x_i=n-s$ for $i\in[r-1]$, $x_{r}=x$ and $x_j=3$ for $j\in\{r+1,\cdots,t\}.$
\end{lemma}
\begin{proof}
Let $\mathbf{c}_t=(c_1,\cdots,c_t)$ be a maximum integer sequence of $\mathcal{S}_t$.
Then $c_i\geq c_{i+1}$ for $i\in[t-1]$.
Since $D<t{n-s\choose 2}$, there is an integer $r\leq t$ such that $c_i=n-s$ for $i\leq r-1$ and $3\leq c_i<n-s$ for $i\in\{r,\cdots,t\}$.
Let $x=c_r$. Then $3\leq x< n-s$.
We need to show $c_i=3$ for each $i\in\{r+1,\cdots,t\}$.
Otherwise, suppose $j$ is the maximum integer of $\{r+1,\cdots,t\}$ with $n-s>c_j>3$.
Let $\mathbf{d}_t=(d_1,\cdots,d_t)$, where $d_i=c_i$ when $i\notin \{r,j\}$,
$d_{r}=c_{r}+1$ and $d_j=c_j-1$.
Then $f(\mathbf{d}_t)\geq f(\mathbf{c}_t)\geq D$, $3\leq d_i< n-s$ for each $i\in[t]$, and $g(\mathbf{c}_t)=g(\mathbf{d}_t)$.
i.e., $\mathbf{d}_t\in \mathcal{S}_t$.
However, $\mathbf{c}_t\prec\mathbf{d}_t$, which contradicts that $\mathbf{c}_t$ is a maximum integer sequence of $\mathcal{S}_t$.
\end{proof}

\begin{lemma} \label{expression-1}
Suppose $t\geq r$, $\mathbf{a}_t\in \mathcal{S}_t$ and $\mathbf{b}_r\in \mathcal{S}_r$. Then $g(\mathbf{b}_r)\leq g(\mathbf{a}_t)$.
\end{lemma}
\begin{proof}
The result holds for $t=r$, so let $t>r$.
W.l.o.g., suppose $\mathbf{a}_t=(a_1,\cdots,a_t)$, where $a_1=\cdots=a_{l-1}=n-s$, $3\leq a_l< n-s$ and $a_{l+1}=\cdots=x_t=3$.
Since $t>r$ and $D<r{n-s\choose 2}$, $l<t$ and $a_t=3$.
Let $\mathbf{c}_{t-1}=(c_1,\cdots,c_{t-1})$, where $c_1=\cdots=c_{l-1}=n-s$, $c_l=a_l+1$ and $c_{l+1}=\cdots=x_{t-1}=3$.
Then $f(\mathbf{c}_{t-1})\geq D$ and $g(\mathbf{c}_{t-1})= g(\mathbf{a}_t)$.
Let $\mathbf{d}_{t-1}\in \mathcal{S}_{t-1}$.
Then $g(\mathbf{c}_{t-1})\geq g(\mathbf{d}_{t-1})$.
By induction on $t-r$, $g(\mathbf{b}_r)\leq g(\mathbf{d}_{t-1})$.
Thus $g(\mathbf{b}_r)\leq g(\mathbf{a}_t)$.
\end{proof}

The following result is easily seen.
\begin{lemma}\label{cho-m}
If $a,b,c$ are positive integers with $c+a-1\geq 2$ and $a+b=c$, then ${c\choose 2}-{a\choose 2}\geq b$.
\end{lemma}

Suppose $X$ is a proper vertex set of $G$.
We use $E(X)$ to denote the set of edges whose ends are in $X$.
For a graph $G$ and $X\subseteq V(G)$, to shrink $X$ is to delete $E(X)$ and then merge the vertices of $X$ into a single vertex.
A partition of the vertex set $V$ is to divide $V$ into some mutual disjoint nonempty sets.
Suppose $\cP=\{V_1,\cdots,V_s\}$ is a partition of $V(G)$.
Then $G/\cP$ is a graph obtained from $G$ by shrinking every $V_i$ into a single vertex.

The spanning tree packing number (STP number) of a graph is the maximum number of edge-disjoint spanning trees contained in the graph.
We use $T(G)$ to denote the number of edge-disjoint spanning trees of $G$.
The following theorem was proved by Nash-Williams and Tutte independently.
\begin{theorem}[\cite{NW}~\cite{T}] \label{N-W-T}
A graph $G$ has at least $k$ edge-disjoint spanning trees if and only if $e(G/\cP)\geq k(|G/\cP|-1)$ for any vertex-partition $\cP$ of $V(G)$.
\end{theorem}
We denote $\tau(G)=\min_{|\cP|\geq2}\frac{e(G/\cP)}{|G/\cP|-1}$.
Then Nash-Williams-Tutte Theorem can be restated as follows.
\begin{theorem} \label{N-T}
$T(G)=k$ if and only if $\lfloor\tau(G)\rfloor=k$.
\end{theorem}

If $\Gamma$ is an extremal $RMC_k$-coloring of $G$,
then we say that $\Gamma$ {\em wastes} $\omega=\sum_{i\in[r]}(|T_i|-2)$ colors,
where $T_1,\cdots,T_r$ are all the nontrivial color-induced trees of $G$.
Thus $rmc_k(G)=m-\omega$.

Suppose that $\Gamma$ is an edge-coloring of $G$ and $v$ is a vertex of $G$.
The {\em nontrivial color degree} of $v$ under $\Gamma$ is denoted by $d^n(v)$,
that is, the number of nontrivial colors appearing on the edges incident with $v$.
\begin{lemma}
Suppose that $\Gamma$ is an $RMC_k$-coloring of $G$ with $k\geq2$.
Then $d^n(v)\geq k$ for every vertex $v$ of $G$.
\end{lemma}
\begin{proof}
Since every two vertices have $k\geq 2$ rainbow monochromatic paths connecting them and $G$ is simple,
every two vertices have at least one nontrivial monochromatic path connecting them,
i.e., $d^n(v)\geq 1$ for each $v\in V(G)$.
Let $e=vu$ be a nontrivial edge.
Then there are $k-1$ rainbow monochromatic paths of order at least three connecting $u$ and $v$.
Since these $k-1$ rainbow monochromatic paths are nontrivial, $d^n(v)\geq k$ for each $v\in V(G)$.
\end{proof}

\section{Existence of $RMC_k$-colorings}

We knew that there exists an $MC_k$-coloring or a $UMC_k$-coloring of $G$ if and only if $G$ is $k$-edge-connected.
It is natural to ask how about $RMC_k$-colorings ?
It is obvious that any cycle of order at least $3$ is $2$-edge-connected,
but it does not have an $RMC_2$-coloring.

We mainly think about simple graphs in this paper,
but in the following result, all graphs may have parallel edges but no loops.
\begin{theorem} \label{SN}
A graph $G$ has an $RMC_k$-coloring if and only if $\tau(G)\geq k$.
\end{theorem}
\begin{proof}
If $G$ has $k$ edge-disjoint spanning trees $T_1,\cdots,T_k$,
then we can color the edges of each $T_i$ by $i$ and color the other edges of $G$ by
colors in $[k]$ arbitrarily. Then the coloring is an $RMC_k$-coloring of $G$.
Therefore, $G$ has an $RMC_k$-coloring when $\tau(G)\geq k$.

We will prove that if there exists an $RMC_k$-coloring of $G$,
then $G$ has $k$ edge-disjoint spanning trees,
i.e., $\tau(G)\geq k$.
Before proceeding to the proof, we need a critical claim as follows.
\begin{claim} \label{IFF}
If $G$ has an $RMC_k$-coloring, then $e(G)\geq k(n-1)$.
\end{claim}
\begin{proof}
Suppose that $\Gamma$ is an extremal $RMC_k$-coloring of $G$ and $G_1,\cdots,G_t$ are all the color-induced trees of $G$ (say $G_i$ is the $i$-induced tree).
If there are two color-induced trees $G_i$ and $G_j$ satisfying that all the three sets $V(G_i)-V(G_j)$, $V(G_j)-V(G_i)$ and $V(G_i)\cap V(G_j)$ are nonempty,
then we use $P(G,\Gamma,i,j)$ to denote the graph $(G-E(G_i\cup G_j))\cup T_1\cup T_2$,
where $T_1$ and $T_2$ are two new trees with $V(T_1)=V(G_i)\cup V(G_j)$ and $V(T_2)=V(G_i)\cap V(G_j)$
(note that $T_1,T_2$ and $G-E(G_i\cup G_j)$ are mutually edge disjoint,
then $P(G,\Gamma,i,j)$ may have parallel edges);
we also use $\Upsilon(G,\Gamma,i,j)$ to denote the edge-coloring of $P(G,\Gamma,i,j)$,
which is obtained from $\Gamma$ by coloring $E(T_1)$ with $i$ and coloring $E(T_2)$ with $j$, respectively.
Then $|G|=|P(G,\Gamma,i,j)|$ and $e(G)=e(P(G,\Gamma,i,j))$.

We claim that $\Upsilon(G,\Gamma,i,j)$ is an $RMC_k$-coloring of $P(G,\Gamma,i,j)$, and we prove it below.
For any two vertices $u,v$ of $G$, if at least one of them is in $V(G)-V(G_i\cup G_j)$,
or one is in $V(G_i)-V(G_j)$ and the other is in $v\in V(G_j)-V(G_i)$,
then none of rainbow monochromatic $uv$-paths of $G$ are colored by $i$ or $j$,
these rainbow monochromatic $uv$-paths of $G$ are kept unchanged.
Thus there are at least $k$ rainbow monochromatic $uv$-paths in $P(G,\Gamma,i,j)$ under $\Upsilon(G,\Gamma,i,j)$;
if both of $u,v$ are in $V(G_i)\cap V(G_j)$,
then there are at least $k-2$ rainbow monochromatic $uv$-paths of $G$ with colors different from $i$ and $j$,
and these rainbow monochromatic $uv$-paths are kept unchanged.
Since $T_1$ and $T_2$ provide two rainbow monochromatic $uv$-paths, one is colored by $i$ and the other is colored by $j$,
there are at least $k$ rainbow monochromatic $uv$-paths in $P(G,\Gamma,i,j)$ under $\Upsilon(G,\Gamma,i,j)$;
if, by symmetry, $u$ and $v$ are in $G_i$ and at most one of them is in $V(G_i)\cap V(G_j)$,
then there are at least $k-1$ rainbow monochromatic $uv$-paths with colors different from $i$ and $j$,
and these rainbow monochromatic $uv$-paths are kept unchanged.
Since $T_1$ provides a monochromatic $uv$-path with color $i$,
there are at least $k$ rainbow monochromatic $uv$-paths in $P(G,\Gamma,i,j)$ under $\Upsilon(G,\Gamma,i,j)$.

We now introduce a simple algorithm on $G$.
Setting $H:=G$ and $\Gamma^*:=\Gamma$.
If there are two color-induced subgraphs $H_i$ and $H_j$ of $H$ satisfying that all the three sets $V(H_i)-V(H_j)$, $V(H_j)-V(H_i)$
and $V(H_i)\cap V(H_j)$ are nonempty, then replace $H$ by $P(H,\Gamma^*,i,j)$ and replace $\Gamma^*$ by $\Upsilon(H,\Gamma^*,i,j)$.

We now show that the algorithm will terminate in a finite steps.
In the $i$th step, let $H=H_i$ and $\Gamma^*=\Gamma_i$, and
let $G^i_1,\cdots,G^i_{t_i}$ be all the color-induced subgraphs of $H_i$ such that
$|G^i_1|\geq |G^i_2|\geq \cdots \geq |G^i_{t_i}|$ (in fact, in each step, each color-induced subgraph is a tree),
and let $l_i=(|G^i_1|,|G^i_2|,\cdots,|G^i_{t_i}|)$ be an integer sequence.
Suppose $H_{i+1}=P(H_i,\Gamma_i,s,t)$, i.e., $H_{i+1}=H_i-E(G^i_s\cup G^i_t)\cup T_1\cup T_2$,
where $V(T_1)=V(G^i_s)\cup V(G^i_t)$ and $V(T_2)=V(G^i_s)\cap V(G^i_t)$.
Then $|T_1|>\max\{|G^i_s|,|G^i_t|\}$. Therefore, $l_i\prec l_{i+1}$.
Since $G$ is a finite graph and $e(H_i)=e(G)$ in each step, the algorithm will terminate in a finite step.

Let $H'$ be the resulting graph and $\Gamma'$ be the resulting $RMC_k$-coloring of $H'$,
and $T'_1,\cdots,T'_r$ be the color-induced trees of $H'$ with $|T'_1|\geq\cdots\geq |T'_r|$.
Then $T'_k$ is a spanning tree of $H'$; otherwise, there is al least one vertex $w$ in $V(G)-V(T_k)$.
Suppose $u\in V(T_k)$. Since $T'_1,\cdots,T'_{k-1}$ provide at most $k-1$ rainbow monochromatic $uw$-paths,
there is a tree of $\{T'_{k+1},\cdots,T'_r\}$, say $T'_a$, containing $u$ and $w$.
Then $V(T'_k)-V(T'_a)\neq \emptyset$; otherwise $|T'_k|<|T'_a|$, a contradiction.
Thus $V(T'_k)-V(T'_a)$, $V(T'_a)\cap V(T'_k)$ and  $V(T'_a)-V(T'_k)$ are nonempty sets,
which contradicts that $H'$ is the resulting graph of the algorithm.
Therefore, there are at least $k$ spanning trees of $H'$, i.e., $e(G)=e(H')\geq k(n-1)$.
\end{proof}

Now, we are ready to prove $\tau(G)\geq k$ by contradiction.
Suppose that $\Gamma$ is an $RMC_k$-coloring of $G$ but $\tau(G)< k$.
By Theorem \ref{N-W-T}, there exists a partition $\mathcal{P}=\{V_1,\cdots,V_t\}$ of $V(G)$ ($|\mathcal{P}|=t\geq2$), such that $e(G/\mathcal{P})<k(|\mathcal{P}|-1)$.
Let $G^*=G/\mathcal{P}$ be the graph obtained from $G$ by shrinking each $V_i$ into a single vertex $v_i$, $1\leq i\leq t$.

Suppose that $\Gamma^*$ is an edge-coloring of $G^*$ obtained from $\Gamma$ by keeping the color of every edge of $G$ not being deleted
(we only delete edges contained in each $V_i$). It is obvious that $\Gamma^*$ is an $RMC_k$-coloring of $G^*$.
However, $e(G^*)<k(|G^*|-1)$, a contradiction to Claim \ref{IFF}. So, $\tau(G)\geq k$.
\end{proof}

We will turn to discuss simple graphs below. Because a simple graph is also a loopless graph,
Theorem \ref{SN} holds for simple graphs. For a connected simple graph $G$, since $1\leq\tau(G)\leq\tau(K_n)=\lfloor\frac{e(K_n)}{n-1}\rfloor=\lfloor\frac{n}{2}\rfloor,$
we have the following result.
\begin{corollary} \label{bound-k}
If $G$ is a simple graph of order $n$ and $G$ has an $RMC_k$-coloring, then $1\leq k\leq \lfloor\frac{n}{2}\rfloor$.
\end{corollary}

By Theorem \ref{SN}, if $\tau(G)\geq k$, a trivial $RMC_k$-coloring of a graph $G$ is a coloring that colors
the edges of the $k$ edge-disjoint spanning trees of $G$ by colors in $[k]$, respectively,
and then colors the other edges trivial.
Since the edge-coloring wastes $k(n-2)$ colors, $rmc_k(G)\geq m-k(n-2)$.
Thus, $m-k(n-2)$ is a lower bound of $rmc_k(G)$ if $G$ has an $RMC_k$-coloring.
\begin{corollary}
If $G$ is a graph with $\tau(G)\geq k$, then $rmc_k(G)\geq m-k(n-2)$.
\end{corollary}

\section{Some graphs with rainbow monochromatic $k$-edge-connection number $m-k(n-2)$}

In this section, we mainly study the graphs with rainbow monochromatic $k$-edge-connection number $m-k(n-2)$ (graphs in the following theorem).
\begin{theorem}\label{RMC-main}
Let $G$ be a graph with $\tau(G)\geq k$.
If $G$ satisfies any of the following properties, then $rmc_k(G)=m-k(n-2)$.
\begin{enumerate}
 \item $G$ is triangle-free;
 \item $diam(G)\geq3$;
 \item $G$ has a cut vertex;
 \item $G$ is not $k+1$-edge-connected.
 \end{enumerate}
\end{theorem}
We will prove this theorem separately by four propositions below (the second result is a corollary of Proposition \ref{RMC-nG}).

\begin{proposition}\label{RMC-mian-1}
If $G$ is a triangle-free graph with $\tau(G)\geq k$, then $rmc_k(G)=m-k(n-2)$.
\end{proposition}
\begin{proof}
By Theorem \ref{RMD-CY}, the result holds for $k=1$.
Therefore, let $k\geq2$ (this requires $n\geq 4$).
Since $G$ is a triangle-free graph, by Tur$\acute{a}$n's Theorem, $e(G)\leq \frac{n^2}{4}$.
Then $k\leq\tau(G)\leq \frac{e(G)}{|G|-1}\leq \frac{n+1}{4}+\frac{1}{4(n-1)}$.
So, $n\geq 4k-1-\frac{1}{n-1}$, i.e., $n\geq 4k-1$.

Suppose $\Gamma$ is an extremal $RMC_k$-coloring of $G$.
If there is a color-induced tree, say $T$, that forms a spanning tree of $G$,
then $\Gamma$ is an extremal $RMC_{k-1}$-coloring restricted on $G-E(T)$
(It is obvious that $\Gamma$ is an $RMC_{k-1}$-coloring restricted on $G-E(T)$.
If $\Gamma$ is not an extremal $RMC_{k-1}$-coloring restricted on $G-E(T)$,
then there is an $RMC_{k-1}$-coloring $\Gamma'$ of $G-E(T)$ such that $|\Gamma(G-E(T))|<|\Gamma'(G-E(T))|$.
Let $\Gamma''$ be an edge-coloring of $G$ obtained from $\Gamma'$ by assigning $E(T)$ with a new color.
Then $\Gamma''$ is an $RMC_k$-coloring of $G$.
However, $|\Gamma(G)|<|\Gamma''(G)|$, a contradiction).
Since $G-E(T)$ is triangle-free, by induction on $k$,
$$rmc_{k-1}(G-E(T))= e(G-E(T))-(k-1)(n-2)=m-k(n-2)-1.$$
Therefore, $rmc_k(G)=1+|\Gamma(G-E(T))|= 1+rmc_{k-1}(G-E(T))= m-k(n-2)$.

Now, suppose that each color-induced tree is not a spanning tree.
We use $\mathcal{S}$ to denote the set of nontrivial color-induced trees of $G$.
We will prove that $\Gamma$ wastes at least $k(n-2)$ colors below.

{\bf Case 1.} There is a vertex $v$ of $G$ such that $d^n(v)= k$.

Suppose that $\mathcal{T}=\{T_1,\cdots,T_k\}$ is the set of the $k$ nontrivial color-induced trees containing $v$.
Since each vertex connects $v$ by at least $k-1\geq1$ nontrivial rainbow monochromatic paths,
$V(G)=\bigcup_{i\in[k]}V(T_i)$. Let $S=\bigcap_{i\in[k]}V(T_i)$ and $S_i=V(T_i)-S$.

For any $i,j\in[k]$, both $S_i-S_j$ and $S_j-S_i$ are nonempty.
Otherwise, suppose $S_i\subseteq S_j$.
Since $T_j$ is not a spanning tree, there is a vertex $u'\in V(G)-V(T_j)$.
Then there are at most $k-2$ nontrivial rainbow monochromatic $u'v$-paths, a contradiction.

According to the above discussion, $S,S_1,\cdots,S_k$ are all nonempty sets.
Moreover, since $k\geq2$, $|V(G)-S|\geq2$.

For each $i\in[k]$ and a vertex $u$ in $S_i$, there is an $i_u\in[k]$ such that $u\notin V(T_{i_u})$.
Furthermore, $u\in V(T_j)$ for each $j\in [k]-\{i_u\}$; for otherwise, there are at most $k-2$ nontrivial rainbow monochromatic
$uv$-paths, which contradicts that $\Gamma$ is an $RMC_k$-coloring of $G$.
Therefore, there are exactly $k-1$ nontrivial rainbow monochromatic $uv$-paths. This implies that $uv$ is a trivial edge of $G$.
Thus, $v$ connects each vertex of $V(G)-S$ by a trivial edge.
Since $G$ is triangle-free, $V(G)-S$ is an independent set.
It is easy to verify that $\mathcal{T}$ wastes
$$\sum_{i\in[k]}(|T_i|-2)=\sum_{i\in[k]}|T_i|-2k=k|S|+(k-1)(n-|S|)-2k=k(n-2)+|S|-n$$ colors.

Let $\mathcal{F}=\mathcal{S}-\mathcal{T}$ (recall that $\mathcal{S}$ is the set of nontrivial trees of $G$).
Since each two vertices of $V(G)-S$ are in at most $k-1$ trees of $\mathcal{T}$ and $V(G)-S$ is an independent set,
there is at least one tree of $\mathcal{F}$ containing them.
Moreover, such a tree contains at least one vertex of $S$.
Suppose that $F_1,\cdots,F_t$ are trees of $\mathcal{F}$ with $|V(F_i)\cap (V(G)-S)|=x_i\geq2$ and $x_1\geq x_2\geq\cdots\geq x_t$.
Let $w_i\in V(F_i)\cap S$ and $W_i=V(F_i)\cap (V(G)-S)\cup \{w_i\}$.
Then $3\leq |W_i|\leq n-|S|+1$ for each $i\in[t]$, and
\begin{align}
\sum_{i\in[t]}{|W_i|-1\choose 2}\geq {n-|S|\choose 2}. \label{trian-dn}
\end{align}
$\mathcal{F}$ wastes at least $\sum_{i\in[t]}(|F_i|-2)\geq \sum_{i\in[t]}(|W_i|-2)$ colors.

For any $i,j\in[k]$, since both $S_i-S_j$ and $S_j-S_i$ are nonempty,
there are at most $k-2$ rainbow monochromatic paths connecting every vertex of $S_i-S_j$
and every vertex of $S_j-S_i$ in $\mathcal{T}$.
Thus there are at least two trees of $\mathcal{F}$ containing the two vertices, i.e., $t\geq2$.

If $k=2$ and $|S|-1=3$, then $\mathcal{F}$ wastes at least two colors, and thus $\Gamma$ wastes at least $k(n-2)$ colors.
Otherwise, $|S|-1\geq 4$. Then by Lemma \ref{expression}, the expression $\sum_{i\in[t]}(|W_i|-2)$,
subjects to $(\ref{trian-dn})$, $n-|S|+1\geq |W_i|\geq 3$ and $t\geq2$,
is minimum when $|W_1|=n-|S|+1$, and $|W_i|=3$ for $i=2,3\cdots,t$.
Then $\mathcal{F}$ wastes at least $n-|S|$ colors,
and thus $\Gamma$ wastes at least $k(n-2)$ colors.

{\bf Case 2.} each vertex $v$ of $G$ has $d^n(v)\geq k+1$.

Suppose $\mathcal{S}=\{T_1,\cdots,T_r\}$ and $|T_i|\geq |T_{i+1}|$ for $i\in[r-1]$.
Since $d^n(v)\geq k+1$ for each vertex $v$ of $G$,
$\sum_{i\in[r]}|T_i|\geq (k+1)n$.

If $r\leq \frac{n}{2}+k$, then $\sum_{i\in[r]}(|T_i|-2)\geq k(n-2)$.
This implies that $\Gamma$ wastes at least $k(n-2)$ colors.
Thus, we consider $r> \frac{n}{2}+k$.

Since each pair of non-adjacent vertices are connected by at least $k$ rainbow monochromatic paths of order at least three,
and each pair of adjacent vertices are connected by at least $k-1$ rainbow monochromatic paths of order at least three,
there are at least $k[{n\choose 2}-e(G)]+(k-1)e(G)= k{n\choose 2}-e(G)$ such paths.
Since each $T_i$ of $\mathcal{S}$ provides ${|T_i|-1\choose 2}$ paths of order at least three, we have
$$\sum_{i\in[r]}{|T_i|-1\choose 2}\geq k{n\choose 2}-e(G).$$
Since $e(G)\leq \frac{n^2}{4}$,
\begin{align}
\sum_{i\in[r]}{|T_i|-1\choose 2}\geq k{n\choose 2}-\frac{n^2}{4}. \label{trian-0}
\end{align}
If $|T_i|=n-1$ for each $i\in[r]$, since $r> \frac{n}{2}+k$, $\Gamma$ wastes $r(n-3)>k(n-2)$ colors.
Thus, we assume that there are some trees of $\mathcal{S}$ with order less than $n-1$.
By Lemma \ref{expression}, there are integers $t,x$ with $t<r$ and $3\leq x\leq n-2$,
such that the expression $\sum_{i\in[r]}(|T_i|-2)$,
subjects to $(\ref{trian-0})$ and $3\leq |T_i|\leq n-1$,
is minimum when $|T_i|=n-1$ for $i\in [t]$, $|T_{t+1}|=x$ and $|T_j|=3$ for $j\in\{t+1,\cdots,r\}$.
By $(\ref{trian-0})$,
\begin{align}
t{n-2\choose 2}+{x-1\choose 2}+r-t-1\geq k{n\choose 2}-\frac{n^2}{4}. \label{trian-1}
\end{align}
This implies that $\Gamma$ wastes at least
\begin{align}
w(\Gamma)=t(n-3)+x-2+r-t-1 \label{trian-2}
\end{align}
colors.

If $t\geq k$, or $t=k-1$ and $x\geq \frac{n}{2}+k-1$,
then $\Gamma$ wastes at least
\begin{align*}
(k-1)(n-3)+x-2+r-k=k(n-2)+(r+x+1-2k-n)\geq k(n-2)
\end{align*}colors.

If $t=k-1$ and $x< \frac{n}{2}+k-1$, then suppose $y$ is a positive integer such that $x+y=\lceil\frac{n}{2}+k-1\rceil$.
Let $z=\lceil\frac{n}{2}+k-1\rceil$.
Recall that $n\geq 4k-1$ and $x\geq 3$, and then $x+z-3\geq 7$.
By Lemma \ref{cho-m}, ${z-1\choose 2}-{x-1\choose 2}\geq y-1$.
We have
\begin{align*}
\sum_{i\in[r]}{|T_i|-1\choose 2}&=(k-1){n-2\choose 2}+{x-1\choose 2}+r-k\\
&\leq (k-1){n-2\choose 2}+{z-1\choose 2}-y+1+r-k\\
&\leq (k-1){n-2\choose 2}+{\frac{n}{2}+k-1\choose 2}-y+1+r-k\\
&=\frac{4k-3}{8}n^2-\frac{8k-7}{4}n+\frac{(k-1)(k+2)}{2}+r-y\\
&=k{n\choose 2}-\frac{n^2}{4}-(\frac{n^2}{8}+\frac{6k-7}{4}n-\frac{(k+2)(k-1)}{2}) +r-y.
\end{align*}
By $(\ref{trian-0})$, we have
$$-(\frac{n^2}{8}+\frac{6k-7}{4}n-\frac{(k+2)(k-1)}{2}) +r-y\geq0,$$
i.e., $r\geq\epsilon+y$, where $\epsilon=\frac{n^2}{8}+\frac{6k-7}{4}n-\frac{(k+2)(k-1)}{2}$.
Then $\Gamma$ wastes
\begin{align*}
\sum_{i\in[r]}(|T_i|-2)&\geq(k-1)(n-3)+x-2+r-k\\
&\geq k(n-2)+(x+y-k+1)-n-k+\epsilon\\
&\geq k(n-2)-\frac{n}{2}-k+\epsilon
\end{align*}
colors. Let
$$h(n)=-\frac{n}{2}-k+\epsilon=\frac{1}{8}[n^2+(12k-18)n-4(k^2+3k-2)].$$
Then $h(n)\geq 0$ when $n\geq \frac{1}{2}(\sqrt{160k^2-384k+292}-12k+18)$.
Thus $h(n)\geq 0$ when $n\geq \frac{k}{2}+9$. Recall that $n\geq 4k-1$, and then $n\geq \frac{k}{2}+9$ holds for $k\geq 3$.
So $\Gamma$ wastes at least $k(n-2)$ colors if $k\geq 3$. If $k=2$, then $h(n)=\frac{1}{8}(n^2+6n-32)$. Since $n\geq 4k-1= 7$, $h(n)\geq 0$.
Therefore, $\Gamma$ wastes at least $k(n-2)$ colors when $k=2$.

If $t\leq k-2$, then the number of trees of order $3$ is at least $r-t-1$. Recall that $n\geq 4k-1\geq 7$ and $k\geq 2$. By $(\ref{trian-1})$,
\begin{align*}
r-t-1&\geq k{n\choose 2}-\frac{n^2}{4}-t{n-2\choose 2}-{x-1\choose 2}\\
&\geq k{n\choose 2}-\frac{n^2}{4}-(k-1){n-2\choose 2}\\
&\geq k(2n-3)+\frac{1}{4}(n^2-10n+12)\\
&\geq k(2n-3)-\frac{9}{4}\geq k(n-2).
\end{align*}
Thus, $\Gamma$ wastes at least $k(n-2)$ colors.
\end{proof}

For a graph $G$, we use $N_{uv}$ to denote the set of common neighbors of $u$ and $v$,
and let $n_{uv}=|N_{uv}|$, $n_G=\min\{n_{uv}:u,v\in V(G)\mbox{ and }u\neq v \}$.
\begin{proposition}\label{RMC-nG}
If $G$ is a graph with $\tau(G)\geq k$, then $rmc_k(G)\leq m-k(n-2)+n_G$.
\end{proposition}
\begin{proof}
Suppose $\Gamma$ is an extremal $RMC_k$-coloring of $G$.
Let $u,v$ be two vertices of $G$ with $n_{uv}=n_G$. Let $V(G)-N[v]-\{u\}=A$, $N_{uv}=C$ and $N(v)-\{u\}=B$.
Then $C\subseteq B$.
Suppose that $\mathcal{T}$ is the set of nontrivial trees containing $u$ and $v$,
$\mathcal{F}$ is the set of nontrivial trees containing $u$
and at least one vertex of $B$ but not $v$,
and $\mathcal{H}$ is the set of nontrivial trees containing $v$
and at least one vertex of $A$ but not $u$.
Thus, $\mathcal{T},\mathcal{F}$ and $\mathcal{H}$ are pairwise disjoint.

The vertex set $A$ is partitioned into $k+1$ pairwise disjoint subsets
$A_0,\cdots,A_k$ (some sets may be empty)
such that every vertex of $A_i$ is in exactly $i$ nontrivial trees of $\mathcal{T}$ for $i\in\{0,\cdots,k-1\}$
and every vertex of $A_k$ is in at least $k$ nontrivial trees of $\mathcal{T}$.
The vertex set $B$ can also be partitioned into $k+1$ pairwise disjoint subsets $B_0,\cdots,B_k$ (some sets may be empty)
such that every vertex of $B_i$ is in exactly $i$ nontrivial trees of $\mathcal{T}$ for $i\in\{0,\cdots,k-1\}$
and every vertex of $B_k$ is in at least $k$ nontrivial trees of $\mathcal{T}$.
Then $\mathcal{T}$ wastes
$$w_1=\Sigma_{T\in\mathcal{T}}(|T|-2)\geq\Sigma_{i=0}^ki(|A_i|+|B_i|)$$ colors.

For every vertex $w$ of $A_i$, since $N(v)\cap A=\emptyset$,
there are at least $k$ nontrivial trees containing $v$ and $w$.
Since there are $i$ such trees in $\mathcal{T}$ for $i\neq k$,
there are at least $k-i$ nontrivial trees connecting $v$ and $w$ in $\mathcal{H}$.
Since every nontrivial tree of $\mathcal{H}$ must contain $v$ and a vertex of $B$, $\mathcal{H}$ wastes
$$w_2=\Sigma_{H\in\mathcal{H}}(|H|-2)\geq\Sigma_{i=0}^k(k-i)|A_i|$$
colors.

Let $C_i=\{w:w\in B_i\cap C\mbox{ and }uw\mbox{ is a trivial edge}\}$.
For each vertex $w$ of $B$, if $w\in B_i-C_i$,
then there are at least $k$ nontrivial trees containing $u$ and $w$;
if $w\in C_i$, there are at least $k-1$ nontrivial trees containing $u$ and $w$.
This implies that each vertex of $B_i-C_i$, $i\in\{0,\cdots,k-1\}$,
is in at least $k-i$ nontrivial trees of $\mathcal{F}$,
and each vertex of $C_i$ is in at least $k-i-1$ nontrivial trees of $\mathcal{F}$.
Now we partition $\mathcal{F}$ into two parts, $\mathcal{F}_1$ and $\mathcal{F}_2$,
such that
$$\mathcal{F}_1=\{F\in \mathcal{F}:V(F)\subseteq B\cup\{u\}\}$$
and $$\mathcal{F}_2=\{F\in \mathcal{F}:V(F)-(B\cup\{u\})\neq \emptyset\}.$$
Then for every $F$ of $\mathcal{F}_1$, $u$ connects a vertex of $C$ in $F$.
Thus, there are at most $|C|-\sum_{i=0}^k|C_i|$ trees in $\mathcal{F}_1$.
Therefore, $\mathcal{F}$ wastes
\begin{align*}
w_3&=\Sigma_{F\in\mathcal{F}}(|F|-2)\\
&\geq\Sigma_{i=0}^k(k-i)|B_i-C_i|+\Sigma_{i=0}^{k-1}(k-i-1)|C_i|-(|C|-\sum_{i=0}^{k-1}|C_i|)\\
&=-|C|+\Sigma_{i=0}^k(k-i)|B_i|
\end{align*}
colors.

According to the above discussion, $\Gamma$ wastes at least
$$w_1+w_2+w_3\geq-|C|+\Sigma_{i=0}^k[k(|A_i|+|B_i|)]= k(n-2)-n_G$$
colors.
Therefore, $rmc_k(G)\leq m-k(n-2)+n_G$.
\end{proof}

If $G$ is not an $s+1$-connected graph, then $n_G\leq s$.
Thus, we have the following result.
\begin{corollary}
If $G$ is a graph with $\tau(G)\geq k$ and $G$ is not $s+1$-connected, then $rmc_k(G)\leq m-k(n-2)+s$.
\end{corollary}

The next theorem decreases this upper bound by one when $s=1$.
\begin{proposition}\label{RMC-vertex-cut}
If $G$ has a cut vertex and $\tau(G)\geq k\geq2$, then $rmc_k(G)=m-k(n-2)$.
\end{proposition}
\begin{proof}
Let $\Gamma$ be an extremal $RMC_k$-coloring of $G$.
Suppose that $a$ is a vertex cut of $G$ and $A_1,\cdots,A_t$ are components of $G-\{a\}$.
Let $w$ be a vertex of $A_1$,
and let $\mathcal{T}=\{T_1,\cdots,T_r\}$ be the set of nontrivial trees connecting $w$
and some vertices of $\bigcup_{i=2}^tA_i$.
Then each $T_i$ contains $a$.
Suppose $\{S_0,S_1,\cdots,S_k\}$ is a vertex partition of $A_1-w$
such that each vertex of $S_i$ is in exactly $i$ nontrivial trees of $\mathcal{T}$ for
$i=0,1\cdots,k-1$ and each vertex of $S_k$ is in at least $k$ nontrivial trees of $\mathcal{T}$.
Since each vertex of $\bigcup_{i=2}^tA_i$ connects $w$ by at least $k$ trees of $\mathcal{T}$, $\mathcal{T}$ wastes
$$\sum_{i\in [r]}(|T_i|-2)\geq k\sum_{i=2}^t|A_i|+\sum_{i=0}^ki|S_i|$$
colors.

Let $\mathcal{F}=\{F_1,\cdots,F_l\}$ be the set of nontrivial trees connecting at least one vertex of $\bigcup_{i=2}^tA_i$
and at least one vertex of $A_1$ but not $w$.
Then $\mathcal{T}\cap \mathcal{F}=\emptyset$.
Since $a$ is a cut vertex of $G$, each $F_i$ of $\mathcal{F}$ contains $a$.
Since $\mathcal{T}$ provides at most $i$ rainbow monochromatic paths connecting every vertex of $S_i$ and every vertex of $\bigcup_{i=2}^tA_i$,
each vertex of $S_i$ is in at least $k-i$ trees of $\mathcal{F}$. Then $\mathcal{F}$ wastes at least
$$\sum_{i\in [l]}(|F_i|-2)\geq \sum_{i=0}^{k}(k-i)|S_i|$$
colors.
Thus, $\Gamma$ wastes at least
$$\sum_{i\in [r]}(|T_i|-2)+\sum_{i\in [l]}(|F_i|-2)\geq k(\sum_{i=2}^t|A_i|+\sum_{i=0}^k|S_i|)= k(n-2)$$
colors, $rmc_k(G)=m-k(n-2)$.
\end{proof}

\begin{proposition}
If $G$ is not a $k+1$-edge-connected graph and $\tau(G)\geq k\geq2$, then $rmc_k(G)=m-k(n-2)$.
\end{proposition}
\begin{proof}
Since $\tau(G)\geq k$, $G$ is $k$-edge-connected.
Thus, $G$ has an edge cut $S$ such that $|S|=k$.
Then $G-S$ has two components, say $D_1$ and $D_2$.
Let $x\in V(D_1)$ and $y\in V(D_2)$.
For an extremal $RMC_k$-coloring of $G$, there are $k$ color-induced trees (say $T_1,\cdots,T_k$) containing $x$ and $y$,
i.e., each $T_i$ contains exactly one edge of $S$.
For each $u\in V(D_1)$, since there are $k$ rainbow monochromatic $uy$-paths,
each path contains exactly one edge of $S$.
Thus each $T_i$ contains $u$.
By the same reason, each $T_i$ contains each vertex of $V_2$.
Therefore, each $T_i$ is a spanning tree of $G$, and so $rmc_k(G)=m-k(n-2)$.
\end{proof}
\begin{proposition} [\cite{CY}]\label{C-n}
If $G$ is a cycle of order $n$, then $mc(\overline{G})\geq e(\overline{G})-\lceil\frac{2n}{3}\rceil$.
\end{proposition}

By Proposition \ref{C-n}, if $P$ is a Hamiltonian path of $K_n$ with $n\geq 4$,
then $mc(G\backslash P)\geq e(G\backslash P)-\lceil\frac{2n}{3}\rceil$.
The following result is obvious.
\begin{corollary}
$rmc_2(K_n)\geq \lfloor\frac{3n^2-13n}{6}\rfloor+2$, $n\geq 4$.
\end{corollary}

{\bf Remark 1:} The above corollary implies that there are indeed some graphs with rainbow
monochromatic $k$-edge-connection number greater that the lower bound.
In fact, for any $k\geq2$ and $s\geq 2$, there exist graphs with rainbow monochromatic $k$-edge-connection number greater than or equal to $m-k(n-2)+s-1$.
We construct the {\em $(k,s)$-perfectly-connected} graphs below.
A graph $G$ is called a $(k,s)$-perfectly-connected graph if $V(G)$ can be partitioned into
$s+1$ parts $\{v\},V_1,\cdots,V_s$, such that $\tau(G[V_i])\geq k$,
$V_1,\cdots,V_s$ induces a corresponding complete $s$-partite graph (call it $K^s$),
and $v$ has precisely $k$ neighbors in each $V_i$.
Since $\tau(G[V_i])\geq k$, each $G[V_i]$ has $k$ edge-disjoint spanning trees (say $T_1^i,\cdots,T_k^i$).
Let the $k$ neighbors of $v$ in $V_i$ be $u_1^i,\cdots,u_k^i$ and let $e_1^i=vu_1^i,\cdots,e_k^i=vu_k^i$.
Let $T_j=\bigcup_{i\in[s]}e_j^i\cup\bigcup_{i\in[s]}T_j^i$ for $j\in\{2,\cdots,k\}$.
Let $\Gamma$ be an edge-coloring of $G$ such that $\Gamma(T_1^i\cup e_1^i)=i$ for $i\in[s]$, $\Gamma(T_j)=s+j-1$
for $j\in\{2,\cdots,k\}$, and the other edges are trivial.
Then $\Gamma$ is an $RMC_k$-coloring of $G$ and $|\Gamma(G)|=m-k(n-2)+s-1$,
and thus $rmc_k(G)\geq m-k(n-2)+s-1$. $\square$

We propose an open problem below.
If the answer for the problem is true, then it will cover our main Theorem \ref{RMC-main}.
\begin{problem}
For an integer $k\geq 2$ and a graph $G$ with $\tau(G)\geq k$, does $rmc_k(G)\leq mc(G)-(k-1)(n-2)$ hold ?
More generally, does $rmc_k(G)\leq rmc_t(G)-(k-t)(n-2)$ hold for any integer $1\leq t<k$ ?
\end{problem}

\section{Random results}

The following result can be found in text books.
\begin{lemma}[\cite{AS},~Chernoff Bound]\label{CHER-B}
If $X$ is a binomial random variable with expectation $\mu$, and $0<\delta<1$,
then
$$Pr[X<(1-\delta)\mu]\leq \exp(-\frac{\delta^2\mu}{2})$$
and if $\delta>0$,
$$Pr[X>(1+\delta)\mu]\leq \exp(-\frac{\delta^2\mu}{2+\delta}).$$
\end{lemma}

Let $p=\frac{\log n+a}{n}$. The authors in \cite{PA} proved that
$$Pr[G(n,p)\mbox{ is connected almost surely}]=
\begin{cases}
1, & a\longrightarrow +\infty;\\
e^{-e^{-a}}, & |a|=O(1);\\
0, & a\longrightarrow -\infty.
\end{cases}$$
Thus, $p=\frac{\log n}{n}$ is the threshold function for $G(n,p)$ being connected.

A sufficient condition for $G(n,p)$ to have an $RMC_k$-coloring almost surely is that
$T(G(n,p))\geq k$ almost surely. For the STP number problem of $G(n,p)$,
Gao et al. proved the following results.
\begin{lemma}[\cite{PXC}]\label{STP-general}
For every $p\in[0,1]$, we have
$$T(G(n,p))=\min\{\delta(G(n,p)),
\lfloor\frac{e(G(n,p))}{n-1}\rfloor\}$$
almost surely.
\end{lemma}

In this section, we denote $\beta=\frac{2}{\log e-\log 2}\approx 6.51778$.
\begin{lemma}[\cite{PXC}]\label{STP-bound}
If
$$p\geq\frac{\beta(\log n-\log\log n/2)+\omega(1)}{n-1},$$
then $T(G(n,p))= \lfloor\frac{e(G(n,p))}{n-1}\rfloor$ almost surely; if
$$p\leq\frac{\beta(\log n-\log\log n/2)-\omega(1)}{n-1},$$
then $T(G(n,p))= \delta(G(n,p))$ almost surely.
\end{lemma}

We knew that $m-k(n-2)$ is a lower bound of $rmc_k(G)$.
Next is an upper bound of $rmc_k(G)$.
Although the upper bound is rough, it is useful for the subsequent proof.

\begin{proposition}\label{rough-bound}
If $G$ is a graph with $\tau(G)\geq k$, then $rmc_k(G)\leq m-(k-1)(n-2)$.
\end{proposition}
\begin{proof}
Since the result holds for $k=1$, we only consider $k\geq 2$.
Suppose $\Gamma$ is an extremal $RMC_k$-coloring of $G$ and
$\mathcal{T}=\{T_1,\cdots,T_r\}$ is the set of nontrivial color-induced trees with $|T_1|\geq \cdots\geq|T_r|$.
Then
\begin{align}
k{n\choose 2}-e(G)\leq \sum_{i\in[r]}{|T_i|-1\choose 2}. \label{RMC-cover11}
\end{align}

{\bf Case 1.} $T_1$ is a spanning tree of $G$.

Then $\Gamma$ is an extremal $RMC_{k-1}$-coloring restricted on $G'=G-E(T_1)$
(this result has been proved in Theorem \ref{RMC-mian-1}).
By induction on $k$,
$$|\Gamma(G')|=rmc_{k-1}(G')\leq e(G')-(k-2)(n-2).$$
Then
\begin{align*}
rmc_k(G)&=1+|\Gamma(G')|= 1+rmc_{k-1}(G')\leq 1+e(G')-(k-2)(n-2)\leq m-(k-1)(n-2).
\end{align*}

{\bf Case 2.} $|T_i|\leq n-1$ for each $i\in[r]$.

By Lemmas \ref{expression} and \ref{expression-1},
the expression $\sum_{i\in[r]}(|T_i|-2)$,
subjects to $(\ref{RMC-cover11})$ and $3\leq |T_i|\leq n-1$,
is minimum when $|T_1|=\cdots=|T_{r-1}|=n-1$ and $|T_{r}|=x+1$,
where $x$ is an integer with $3\leq x+1\leq n-2$.

If $r\leq k-1$, then $\sum_{i\in[r]}{|T_i|-1\choose 2}< (k-1){n-2\choose 2}<k{n\choose 2}-e(G)$,
a contradiction to (\ref{RMC-cover11}).

If $r>k$, then $\Gamma$ wastes at least $k(n-3)\geq (k-1)(n-2)$ colors.
Thus $rmc_k(G)\leq m-(k-1)(n-2)$.

If $r= k$, then
$$(k-1){n-2\choose 2}+{x\choose 2}\geq k{n\choose2}-e(G).$$
So, $x^2-x-\alpha\geq0$, where
$$\alpha=2[{n\choose 2}+(2n-3)(k-1)-e(G)]=2[(2n-3)(k-1)+e(\overline{G})].$$
The inequality holds when $x\geq \frac{1+\sqrt{1+4\alpha}}{2}\geq \sqrt{\alpha}$.
Thus, $\Gamma$ wastes at least
$$\Sigma_{i\in[k]}(|T_i|-2)=(k-1)(n-2)+x-1\geq(k-1)(n-2)+\sqrt{\alpha}-1.$$
Since $k\geq 2$, $\sqrt{\alpha}\geq 1$.
Thus $rmc_k(G)\leq m-(k-1)(n-2)$.
\end{proof}

\begin{theorem}
Let $k=k(n)$ be an integer such that $\lfloor\frac{n}{2}\rfloor> k\geq 1$
and let $rmc_k(K_n)>f(n)\geq k(n-1)$.
Then
$$p=
\begin{cases}
\frac{f(n)+kn}{n^2}, & f(n)\geq O(n\log n)\mbox{ and }k=o(n);\\
\min\{\frac{k}{n},\frac{\log n}{n}\}, & f(n)=o(n\log n)\mbox{ and }k=o(n);\\
1,& k=O(n)\mbox{ and }f(n)<rmc_k(K_n).
\end{cases}$$
is a sharp threshold function for the property $rmc_k(G(n, p)) \geq f (n)$.
\end{theorem}
\begin{proof}
Let $c$ be a positive constant and let $E(||G(n,cp)||)$ be the expectation of the number of edges in $G(n,cp)$.
Then
$$E(||G(n,cp)||)=
\begin{cases}
\frac{c(n-1)}{2n}f(n)+\frac{c\cdot k(n-1)}{2}, & f(n)\geq O(n\log n)\mbox{ and }k=o(n);\\
\frac{c\cdot k(n-1)}{2}, & f(n)=o(n\log n),k=o(n)\mbox{ and }k>\log n;\\
\frac{c \log n(n-1)}{2}, & f(n)=o(n\log n),k=o(n)\mbox{ and }k\leq\log n;\\
c{n\choose 2},& k=O(n)\mbox{ and }f(n)<rmc_k(K_n).
\end{cases}$$
By Lemma \ref{CHER-B}, both inequalities
\begin{align*}
Pr[||G(n,cp)||<\frac{1}{2}E(||G(n,cp)||)]\leq \exp(-\frac{1}{8}E(||G(n,cp)||))= o(1)
\end{align*}
and
\begin{align*}
Pr[||G(n,cp)||>\frac{3}{2}E(||G(n,cp)||)]\leq \exp(-\frac{1}{10}E(||G(n,cp)||))= o(1)
\end{align*}
hold for each $p$.

{\bf Case 1.} $k=O(n)$, i.e., there is an $l\in \mathbb{R}^+$ such that $l\cdot n\leq k<\lfloor\frac{n}{2}\rfloor$.

Since $G(n,p)=K_n$, $rmc_k(G(n,p))\geq f(n)$ always holds.
On the other hand, we have
$$||G(n,l\cdot p)||\leq\frac{3}{2}E(||G(n,l\cdot p)||)=\frac{3l}{2}\cdot {n\choose 2}<k(n-2)$$
almost surely.
By Claim \ref{IFF}, $G(n,l\cdot p)$ does not have $RMC_k$-colorings almost surely.

{\bf Case 2.} $k=o(n)$.

{\bf Case 2.1.} $f(n)\geq O(n\log n)$.

Then, there is an $s\in \mathbb{R}^+$ and $f(n)\geq s\cdot n\log n$.
Let
$$c_1=
\begin{cases}
\beta+1, & s\geq1;\\
\frac{\beta+1}{s}, & 0<s<1.
\end{cases}$$
Since $f(n)\geq s\cdot n\log n$, we have
\begin{align*}
c_1p\geq \frac{(\beta+1)(\log n+kn)}{n}\geq \frac{\beta(\log n-\log\log n/2)+\omega(1)}{n-1}. \label{pre-used}
\end{align*}
Since
$$||G(n,c_1p)||\geq
\frac{1}{2}E(||G(n,c_1p)||)=
\frac{\beta+1}{2}\cdot\frac{n-1}{2n}f(n)+\frac{k(n-1)(\beta+1)}{4}$$
almost surely, by Lemma \ref{STP-bound}, $T(G(n,c_1p))=\lfloor\frac{||G(n,c_1p)||}{n-1}\rfloor >k$ almost surely,
i.e., $G(n,c_1p)$ has $RMC_k$-colorings almost surely.
Therefore,
\begin{align*}
rmc_k(G(n,c_1p))&\geq ||G(n,c_1p)||-k(n-2)\\
&\geq\frac{\beta+1}{2}\cdot\frac{n-1}{2n}f(n)+\frac{k(n-1)(\beta+1)}{4}-k(n-2)\\
&>\frac{(\beta+1)(n-1)}{4n}f(n)\\
&>f(n)
\end{align*}
almost surely.

Let $c_2=\frac{2}{3}$. Then
\begin{align*}
||G(n,c_2p)||&\leq \frac{3}{2}E(||G(n,c_2p)||)\\
&\leq\frac{3c_2}{2}\cdot\frac{n-1}{2n}f(n)+\frac{3c_2}{2}\cdot\frac{k(n-1)}{2}\\
&<\frac{1}{2}[f(n)+k(n-1)]
\end{align*}
almost surely.
Thus, either $G(n,c_2p)$ does not have $RMC_k$-colorings almost surely,
or $rmc_k(G(n,c_2p))<||G(n,c_2p)||-(k-1)(n-2)<\frac{1}{2}f(n)$ almost surely
(recall that $rmc_k(G)\leq m-(k-1)(n-2)$ by Proposition \ref{rough-bound}).

{\bf Case 2.2.} $f(n)=o(n\log n)$.

If $k\leq \log n$, then $p=\frac{\log n}{n}$.
Let $c_1=\beta+1$ and $c_2=\frac{1}{2}$ be two constants.
Since $c_1p>\frac{(\beta+1)\log n}{n}\geq \frac{\beta(\log n-\log\log n/2)+\omega(1)}{n-1}$,
by Lemma \ref{STP-bound}, $T(G(n,c_1p))=\lfloor\frac{||G(n,c_1p)||}{n-1}\rfloor$ almost surely.
Since
$$||G(n,c_1p)||\geq \frac{1}{2}E(||G(n,c_1p)||)=\frac{\log n(n-1)(\beta+1)}{4}$$
almost surely, $T(G(n,c_1p))\geq \log n\geq k$ almost surely,
i.e., $G(n,c_1p)$ has $RMC_k$-coloring almost surely.
Therefore,
\begin{align*}
rmc_k(G(n,c_1p))&\geq ||G(n,c_1p)||-k(n-2)\\
&\geq\frac{\log n(n-1)(\beta+1)}{4}-k(n-2)\\
&\geq \frac{3\log n(n-1)}{4}>f(n)
\end{align*}
almost surely. For $G(n,c_2p)$, since $c_2p=\frac{\log n}{2n}$, $G(n,c_2p)$ is not connected almost surely,
i.e., $G(n,c_2p)$ does not have $RMC_k$-colorings almost surely.

If $k>\log n$ and $k=o(n)$, then $p=\frac{k}{n}$.
Let $c_1=\beta+1$ and $c_2=1$.
Then $c_1p=\frac{(\beta+1)k}{n}>\frac{(\beta+1)\log n}{n}\geq \frac{\beta(\log n-\log\log n/2)+\omega(1)}{n-1}$,
i.e., $T(G(n,c_1p))=\lfloor\frac{||G(n,c_1p)||}{n-1}\rfloor$ almost surely.
Since $||G(n,c_1p)||\geq \frac{1}{2}E(||G(n,c_1p)||)=\frac{k(n-1)(\beta+1)}{4}$
almost surely, $T(G(n,c_1p))\geq k$ almost surely, i.e., $G(n,c_1p)$ has $RMC_k$-colorings almost surely.
Thus
\begin{align*}
rmc_k(G(n,c_1p))\geq ||G(n,c_1p)||-k(n-2)>\frac{3}{4}k(n-1)>\frac{3}{4}(n-1)\log n>f(n)
\end{align*}
almost surely.
For $G(n,c_2p)$, since $||G(n,c_2p)||\leq \frac{3}{2}E(||G(n,c_2p)||)=\frac{3}{4}k(n-1)<k(n-2)$
almost surely.
By Claim \ref{IFF}, $G(n,c_2p)$ does not have $RMC_k$-colorings almost surely.
\end{proof}

{\bf Remark 2.} Since $rmc_k(G)=rmc_k(K_n)$ if and only if $G=K_n$,
we only concentrate on the case $1\leq f(n)< rmc_k(K_n)$.
If $n$ is odd, then $G$ has $RMC_{\lfloor\frac{n}{2}\rfloor}$-colorings if and only if $G=K_n$.
So, we are not going to consider the case $k=\lfloor\frac{n}{2}\rfloor$.
$\square$

\end{document}